\ifdefined\pdfoutput
  \pdfoutput=1
\fi
\documentclass[10pt]{article}
\usepackage[T1]{fontenc}
\usepackage{lmodern}
\usepackage[a4paper,margin=1in]{geometry}
\usepackage{amsmath,amssymb,amsthm,mathtools}
\usepackage[english]{babel}
\usepackage{microtype}
\usepackage[numbers,sort&compress]{natbib}
\usepackage{parskip}
\usepackage{xcolor}
\usepackage[colorlinks=true,linkcolor=blue!55!black,citecolor=blue!55!black,urlcolor=blue!55!black]{hyperref}
\usepackage[nameinlink,capitalise]{cleveref}
\newtheorem{theorem}{Theorem}[section]
\newtheorem{lemma}[theorem]{Lemma}
\newtheorem{proposition}[theorem]{Proposition}
\newtheorem{corollary}[theorem]{Corollary}
\theoremstyle{definition}
\newtheorem{remark}[theorem]{Remark}

\crefname{theorem}{Theorem}{Theorems}
\crefname{lemma}{Lemma}{Lemmas}
\crefname{proposition}{Proposition}{Propositions}
\crefname{corollary}{Corollary}{Corollaries}
\crefname{remark}{Remark}{Remarks}
\crefname{section}{\S}{Sections}

\AddToHook{env/theorem/begin}{\crefalias{section}{theorem}}
\AddToHook{env/lemma/begin}{\crefalias{theorem}{lemma}}
\AddToHook{env/proposition/begin}{\crefalias{theorem}{proposition}}
\AddToHook{env/corollary/begin}{\crefalias{theorem}{corollary}}
\AddToHook{env/remark/begin}{\crefalias{theorem}{remark}}

\newcommand{\Z}{\mathbb Z}
\newcommand{\rad}{\operatorname{rad}}
\newcommand{\lpf}{\operatorname{lpf}}
\newcommand{\round}{\operatorname{round}}
\newcommand{\OddRound}{\operatorname{OddRound}}
\newcommand{\EvenRound}{\operatorname{EvenRound}}
\newcommand{\NextPrime}{\operatorname{NextPrime}}
\newcommand{\Fingerprint}{\mathcal F}
\newcommand{\Scale}{\mathcal S}

\title{Cyclotomic Prime Extractors}
\author{Joseph M. Shunia\thanks{Independent researcher, Ann Arbor, Michigan, United States. Email: \texttt{jshunia@gmail.com}.}}
\date{September 2026}

\hypersetup{
  pdftitle={Cyclotomic Prime Extractors},
  pdfauthor={Joseph M. Shunia},
  pdfsubject={Local and Archimedean prime extraction from cyclotomic values},
  pdfkeywords={cyclotomic polynomials, p-adic valuation, radical, least prime factor, prime successor, Mellin transform}
}

\begin{document}
\maketitle

\begin{abstract} \noindent
We develop explicit prime-recovery formulas from the values $\Phi_n(2)$ of cyclotomic polynomials.
Binary divisibility patterns detect repeated prime factors and identify the least prime divisor of a squarefree index,
while small corrections to $\log_2\Phi_n(2)$ allow successive recovery of the distinct prime factors.
Specializing the index gives identities for prime products and the least prime above a given integer.
\\[2mm]
The same mechanism extends from finite factorizations to an infinite prime sequence: a normalized limit of cyclotomic
values along the odd primorials defines a real constant $\Omega=0.25061403238015047218\ldots$, from which every odd
prime can be recovered by a recursive rounding rule.
\end{abstract}

\begingroup\small
\par
\endgroup

\section{Introduction}
\label{sec:introduction}

Which parts of the factorization of $n$ can be read directly from the integer
$\Phi_n(2)$? A small example suggests that very elementary observations may
suffice. The value
\[
\Phi_{15}(2)=151=(10010111)_2
\]
ends in three binary ones, and $3$ is the least prime factor of $15$.
Equivalently, $151+1=2^3\cdot19$. The length of this final run is not an
accident: for every squarefree index $n>1$, it is exactly the least prime
factor.
The neighboring integer $\Phi_n(2)-1$ records a different feature of the
index, namely the part left after removing one copy of each prime divisor.

To state these observations, write $\rad(n)$ for the product of the distinct
prime divisors of $n$, let $\lpf(n)$ be the least prime factor of $n>1$, and
let $\nu_2$ denote the $2$-adic valuation. The following identities are the
starting point of the paper.

\begin{theorem}[Cyclotomic prime extractors]
\label{thm:main-intro}
For every integer $n>1$,
\begin{equation}
\label{eq:intro-minus}
\nu_2\!\left(\Phi_n(2)-1\right)=\frac{n}{\rad(n)},
\end{equation}
and
\begin{equation}
\label{eq:intro-plus}
\nu_2\!\left(\Phi_n(2)+1\right)=
\begin{cases}
\lpf(n),&n\text{ squarefree},\\
1,&n\text{ nonsquarefree}.
\end{cases}
\end{equation}
\end{theorem}

Thus two binary patterns distinguish repeated prime powers from
squarefree prime support. For an odd integer $C=\Phi_n(2)$,
$\nu_2(C-1)$ is the position of the first nonzero binary digit above the units
digit, whose position is $0$, while $\nu_2(C+1)$ counts the trailing ones.
For $n=15$, these valuations are $1$ and $3$; for $n=12$, the value
$\Phi_{12}(2)=13=(1101)_2$ returns $2$ and $1$, detecting the repeated factor.
Once the radical has been recovered, successive applications of the plus
identity to squarefree quotients list its prime divisors in increasing order.

The proofs explain why these features of the value retain information about
the index. Put $r=\rad(n)$ and $a=n/r$. The classical reduction
$\Phi_n(X)=\Phi_r(X^a)$ separates the multiplicities from the distinct prime
divisors, and $\Phi_r(Y)$ begins with $1-\mu(r)Y$. At $X=2$, this first term
forces the valuation in \eqref{eq:intro-minus}. The plus identity uses a
second feature of the M\"obius product: the divisor $1$ contributes the
factor $2-1=1$, leaving the least prime divisor as the first relevant factor
modulo a suitable power of $2$.

An extractor for a given index becomes a prime-sequence identity when the
index is chosen to have a useful prime support. For example, the prime
divisors of $n!$ are precisely the primes at most $n$, so
\begin{equation}
\label{eq:intro-primorial}
\prod_{p\le n}p
=\frac{n!}{\nu_2(\Phi_{n!}(2)-1)}
\qquad(n\ge2).
\end{equation}
To move beyond $n$, use the central-binomial quotient
\begin{equation}
\label{eq:intro-Rn}
R_n:=\frac{\binom{2n}{n}}
{\gcd\!\left(\binom{2n}{n},(n!)^2\right)}
=\prod_{n<p\le2n}p,
\end{equation}
introduced in \cite{ShuniaPrimeInterval}. Its least prime divisor is the
least prime strictly greater than $n$, denoted by $\NextPrime(n)$. Hence
\begin{equation}
\label{eq:intro-next}
\NextPrime(n)=\nu_2\!\left(\Phi_{R_n}(2)+1\right)
\qquad(n\ge1).
\end{equation}
In particular, the next prime can be expressed in terms of the current prime
value, without its position in the prime sequence.

There is a second way to read the same cyclotomic value. Factoring the leading
powers of $2$ from the M\"obius product gives
\begin{equation}
\label{eq:intro-log}
\log_2\Phi_n(2)
=\varphi(n)+
\sum_{d\mid r}\mu(r/d)\log_2(1-2^{-ad}).
\end{equation}
The integer $\varphi(n)$ describes the main size of the value. The correction
contains the prime support in a more delicate form: its terms decrease
exponentially with $ad$. After the contribution of the already recovered
primes is removed, the next prime supplies the largest remaining term.
The issue is whether all later terms together can change the answer obtained
by rounding its logarithmic scale. We prove that they cannot, and obtain
explicit reconstruction rules for nonsquarefree, odd squarefree, and even
squarefree indices. The local exponent $a$ in \eqref{eq:intro-minus} is thus
also the first real correction scale in \eqref{eq:intro-log}.

The estimates for odd squarefree indices depend on the next missing prime,
not on the number of primes still to come. This permits a further step:
letting the index range through the odd primorials produces a convergent
logarithmic constant $\Omega$, while preserving the rounding rule. The same
subtraction procedure therefore recovers every odd prime from the limit.
Explicit tail bounds quantify both its convergence and the precision needed
to decode it. Finally, replacing $2^{-1}$ by $e^{-t}$ gives a logarithmic
family. Its Mellin transform turns the integer dilations into an odd
M\"obius Dirichlet series, yielding the zeta quotient in
\cref{prop:Mellin}.

The point of these formulas is the explicit recovery of arithmetic data from
particular features of a cyclotomic value. They do not give efficient
factorization or prime-generation algorithms: the integer values can be
enormous, and the real formulas require enough precision to resolve very
small residuals. We keep this distinction explicit when discussing the
reconstruction procedures.

The M\"obius product, radical reduction, and logarithmic normalization are
classical. Pomerance and Rubinstein-Salzedo give the radical reduction,
logarithmic dominance estimates, and the first-gap formula in
\cite[Proposition~2.5, Lemma~4.1, Theorem~4.2, and
Proposition~6.4]{PomeranceRubinsteinSalzedo}; their results also show that
$\Phi_n(2)$ determines $n$ except for the coincidence
$\Phi_2(2)=\Phi_6(2)=3$. Our purpose here is to give the explicit valuation
and rounding rules, their prime-sequence specializations, and the passage
from finite reconstruction to a decodable limit. To our knowledge, the paired
valuation identities and these explicit peeling formulas have not previously
appeared in this form. General background is available in
\cite{Apostol,Washington}, and coefficient formulas are surveyed in
\cite{HerreraMoree}.

Sections~\ref{sec:preliminaries} and~\ref{sec:local} establish the polynomial
and local identities. Section~\ref{sec:successor} applies them to prime sets.
Section~\ref{sec:logarithmic} develops real reconstruction, and
Section~\ref{sec:infinite-prime-constant} passes to the infinite limit and its
Mellin family.

\section{Radical reduction and the first cyclotomic scale}
\label{sec:preliminaries}

We write $\mu$ for the M\"obius function, $\varphi$ for Euler's totient
function, and
\[
\rad(n)=\prod_{p\mid n}p.
\]
For $n>1$, let $\lpf(n)$ be the least prime factor of $n$. For a prime
$\ell$ and a nonzero integer $m$, $\nu_\ell(m)$ denotes the exponent of
$\ell$ in $|m|$, normalized by $\nu_\ell(\ell)=1$. Unadorned logarithms are
natural logarithms.

The first step is to separate the two pieces of the index: its distinct prime
divisors and their multiplicities. Radical reduction puts the first into a
squarefree cyclotomic polynomial and the second into the exponent of its
argument. We recall the short proof, since this separation will be used in
both the local and the real arguments.

\begin{proposition}[M\"obius product and radical reduction]
\label{prop:radical-reduction}
For every $n>1$,
\begin{equation}
\label{eq:mobius-product}
\Phi_n(X)=\prod_{d\mid n}(X^d-1)^{\mu(n/d)}
\end{equation}
and
\begin{equation}
\label{eq:radical-reduction}
\Phi_n(X)
=
\Phi_{\rad(n)}\!\left(X^{n/\rad(n)}\right).
\end{equation}
\end{proposition}

\begin{proof}
M\"obius inversion applied to
$X^n-1=\prod_{d\mid n}\Phi_d(X)$ gives \eqref{eq:mobius-product}. If
$\mu(n/d)\ne0$, then $n/d$ is squarefree, so $n/\rad(n)$ divides $d$.
Writing $d=(n/\rad(n))e$ in the nonzero factors gives
\[
\Phi_n(X)
=
\prod_{e\mid\rad(n)}
\left(X^{en/\rad(n)}-1\right)^{\mu(\rad(n)/e)},
\]
which is \eqref{eq:radical-reduction}.
\end{proof}

\begin{lemma}[First term at the origin]
\label{lem:first-term}
If $r>1$ is squarefree, then for some $H_r(Y)\in\Z[Y]$,
\begin{equation}
\label{eq:first-term}
\Phi_r(Y)=1-\mu(r)Y+Y^2H_r(Y).
\end{equation}
\end{lemma}

\begin{proof}
Because $r>1$, \eqref{eq:mobius-product} may be written
\[
\Phi_r(Y)=\prod_{d\mid r}(1-Y^d)^{\mu(r/d)}.
\]
Modulo $Y^2$, all factors with $d>1$ are $1$, while the factor with
$d=1$ is $1-\mu(r)Y$. Since $\Phi_r(Y)$ is a polynomial over $\Z$, the
claimed form follows.
\end{proof}

\section{Paired local extractors}
\label{sec:local}

Subtracting the constant term leaves a polynomial whose lowest power is
$X^{n/\rad(n)}$ and whose coefficient there is a unit. At a prime dividing
the argument, higher powers cannot cancel this term. The first local identity
therefore holds more generally than evaluation at $2$.

\begin{theorem}[Local radical extractor]
\label{thm:local-radical}
Let $n>1$, let $\ell$ be prime, and let $x\in\Z\setminus\{0\}$ with
$\ell\mid x$. Then
\begin{equation}
\label{eq:local-radical}
\nu_\ell\!\left(\Phi_n(x)-1\right)
=
\frac{n}{\rad(n)}\nu_\ell(x).
\end{equation}
\end{theorem}

\begin{proof}
Applying \cref{prop:radical-reduction,lem:first-term} with
$r=\rad(n)$ and $a=n/\rad(n)$ gives
\[
\Phi_n(x)-1
=
x^a\bigl(-\mu(r)+x^aH_r(x^a)\bigr).
\]
The factor in parentheses is an $\ell$-adic unit, so its valuation is zero.
\end{proof}

Taking $x=2$ gives the first extractor.

\begin{corollary}[The $2$-adic radical identity]
\label{cor:minus}
For every $n>1$,
\begin{equation}
\label{eq:minus}
\nu_2\!\left(\Phi_n(2)-1\right)=\frac{n}{\rad(n)}.
\end{equation}
Consequently,
\begin{equation}
\label{eq:radical-formula}
\rad(n)=\frac{n}{\nu_2(\Phi_n(2)-1)}.
\end{equation}
\end{corollary}

A factorial provides an immediate application: its radical retains one copy
of every prime at most the factorial's argument. Thus the valuation formula
also expresses a primorial without an explicit product over primes.

\begin{corollary}[Primorial identity]
\label{cor:primorial}
For every $n\ge2$,
\begin{equation}
\label{eq:primorial}
\prod_{p\le n}p
=
\frac{n!}{\nu_2(\Phi_{n!}(2)-1)}.
\end{equation}
\end{corollary}

\begin{proof}
Apply \eqref{eq:radical-formula} to $n!$ and use
$\rad(n!)=\prod_{p\le n}p$.
\end{proof}

For instance, $6!=720$ has radical $2\cdot3\cdot5=30$, so
\[
\nu_2(\Phi_{720}(2)-1)=24
\]
and \eqref{eq:primorial} returns $720/24=30$.

For a squarefree index, the minus valuation is always $1$, so it does not
distinguish the prime divisors. To find the first one, consider
$\Phi_n(2)+1$. Here the choice of $2$ matters: in the M\"obius product the
factor associated with the divisor $1$ is $2-1=1$. The first remaining
divisor is the least prime, and reduction modulo $2^{p+1}$ isolates its
contribution.

\begin{theorem}[Least-prime and squarefreeness extractor]
\label{thm:plus}
Let $n>1$. If $n$ is nonsquarefree, then
\begin{equation}
\label{eq:plus-nonsquare-congruence}
\Phi_n(2)\equiv1\pmod4.
\end{equation}
If $n$ is squarefree and $p=\lpf(n)$, then
\begin{equation}
\label{eq:plus-square-congruence}
\Phi_n(2)\equiv2^p-1\pmod{2^{p+1}}.
\end{equation}
Hence
\begin{equation}
\label{eq:plus}
\nu_2\!\left(\Phi_n(2)+1\right)
=
\begin{cases}
\lpf(n),&n\text{ squarefree},\\
1,&n\text{ nonsquarefree}.
\end{cases}
\end{equation}
\end{theorem}

\begin{proof}
If $n$ is nonsquarefree, then $n/\rad(n)\ge2$, and
\cref{prop:radical-reduction,lem:first-term} gives
$\Phi_n(2)\equiv1\pmod4$.

Now suppose that $n$ is squarefree and put $p=\lpf(n)$. In the group of odd
units modulo $2^{p+1}$, \eqref{eq:mobius-product} gives
\[
\Phi_n(2)=\prod_{d\mid n}(2^d-1)^{\mu(n/d)}.
\]
The only divisors of $n$ not exceeding $p$ are $1$ and $p$. For $d>p$,
$2^d-1\equiv-1\pmod{2^{p+1}}$, and
\[
\sum_{\substack{d\mid n\\d>p}}\mu(n/d)
=-\mu(n)-\mu(n/p)=0.
\]
The $d=1$ factor is $1$, while the $d=p$ factor is congruent to
$2^p-1$ even if its exponent is $-1$, because
$(2^p-1)^2\equiv1\pmod{2^{p+1}}$. This proves
\eqref{eq:plus-square-congruence}; both cases of \eqref{eq:plus} follow.
\end{proof}

\begin{corollary}[Squarefreeness, primality, and factor stripping]
\label{cor:squarefree-prime}
For every $n>1$,
\[
n\text{ is squarefree}
\iff
\nu_2(\Phi_n(2)+1)>1,
\]
and
\[
n\text{ is prime}
\iff
\nu_2(\Phi_n(2)+1)=n.
\]
If $n$ is squarefree, repeatedly applying \eqref{eq:plus} and dividing the
index by the returned valuation lists all prime divisors of $n$ in increasing
order.
\end{corollary}

\begin{proof}
The first equivalence is immediate from \cref{thm:plus}. In the squarefree
case, the valuation is $\lpf(n)$, which equals $n$ exactly when $n$ is prime;
in the nonsquarefree case it is $1<n$. The stripping statement follows by
iteration.
\end{proof}

For an arbitrary index, \cref{cor:minus} first supplies its radical, and
\cref{cor:squarefree-prime} can then be applied to that squarefree index.
The remaining question is how to choose an index whose prime divisors are
exactly the primes we wish to recover.

\section{From extractors to the prime sequence}
\label{sec:successor}

\subsection{The least prime above an integer}

The central binomial coefficient contains every prime in $(n,2n]$, each
once, together with factors from smaller primes. The denominator in
\eqref{eq:intro-Rn} removes precisely those smaller factors. The following
identity is proved in \emph{Prime-Interval Algebras}
\cite[Proposition~2.1]{ShuniaPrimeInterval}.

\begin{lemma}[Central-binomial interval product]
\label{lem:Rn}
For every $n\ge1$,
\begin{equation}
\label{eq:Rn}
R_n
=
\frac{\binom{2n}{n}}
{\gcd\!\left(\binom{2n}{n},(n!)^2\right)}
=
\prod_{n<p\le2n}p.
\end{equation}
\end{lemma}

Let $\NextPrime(n)$ denote the least prime strictly greater than $n$.
Bertrand's theorem ensures that \eqref{eq:Rn} is nonempty
\cite{HardyWright}.

\begin{theorem}[Cyclotomic successor-prime identity]
\label{thm:next-prime}
For every $n\ge1$,
\begin{equation}
\label{eq:next-prime}
\NextPrime(n)=\nu_2\!\left(\Phi_{R_n}(2)+1\right).
\end{equation}
In particular, if $p_k$ is the $k$th prime, then
\begin{equation}
\label{eq:prime-recurrence}
p_{k+1}
=
\nu_2\!\left(\Phi_{R_{p_k}}(2)+1\right).
\end{equation}
\end{theorem}

\begin{proof}
By \cref{lem:Rn}, $R_n$ is squarefree and its least prime factor is
$\NextPrime(n)$. Apply \cref{thm:plus}.
\end{proof}

Starting with $p_1=2$, the recurrence uses each prime as the input for the
next step. The interval quotient supplies a nonempty squarefree index, and
the plus valuation selects its smallest prime divisor.

For example,
\[
R_{10}=11\cdot13\cdot17\cdot19=46189,
\]
so the plus extractor reads the first factor directly:
\[
\nu_2(\Phi_{46189}(2)+1)=11=\NextPrime(10).
\]

\subsection{Factorial indices and dyadic prime shells}
\label{sec:factorial-shells}

The same interval product can be recovered by comparing two primorials.
Passing from $n!$ to $(2n)!$ introduces exactly the primes in $(n,2n]$ into
the radical. Substituting the local radical formula into this observation
gives a recurrence for the factorial-index valuations.

\begin{proposition}[Factorial-index recurrence]
\label{prop:factorial-recurrence}
For every $n\ge2$,
\begin{equation}
\label{eq:factorial-recurrence}
\nu_2\!\left(\Phi_{(2n)!}(2)-1\right)
=
n!\,\gcd\!\left(\binom{2n}{n},(n!)^2\right)
\nu_2\!\left(\Phi_{n!}(2)-1\right).
\end{equation}
Moreover,
\begin{equation}
\label{eq:R-from-valuations}
R_n
=
\frac{(2n)!\,\nu_2(\Phi_{n!}(2)-1)}
{n!\,\nu_2(\Phi_{(2n)!}(2)-1)}.
\end{equation}
\end{proposition}

\begin{proof}
By \eqref{eq:primorial} and \eqref{eq:Rn},
\begin{align*}
\frac{\nu_2(\Phi_{(2n)!}(2)-1)}
{\nu_2(\Phi_{n!}(2)-1)}
&=
\frac{(2n)!}{n!R_n}\\
&=
n!\,\gcd\!\left(\binom{2n}{n},(n!)^2\right).
\end{align*}
This proves \eqref{eq:factorial-recurrence}, and rearranging the first line
gives \eqref{eq:R-from-valuations}.
\end{proof}

At powers of two, consecutive valuations therefore recover the complete prime
shell
\begin{equation}
\label{eq:dyadic-shell}
\prod_{2^j<p\le2^{j+1}}p
=
\frac{(2^{j+1})!\,\nu_2(\Phi_{(2^j)!}(2)-1)}
{(2^j)!\,\nu_2(\Phi_{(2^{j+1})!}(2)-1)}
\qquad(j\ge1).
\end{equation}
The product is squarefree, so repeated use of the plus extractor lists the
primes in each shell in increasing order. Iterating over $j\ge1$ recovers
all odd primes; adjoining $2$ gives the complete prime sequence.

The first two shells make the mechanism concrete. The minus extractor gives
\[
\nu_2(\Phi_{2!}(2)-1)=1,
\qquad
\nu_2(\Phi_{4!}(2)-1)=4,
\qquad
\nu_2(\Phi_{8!}(2)-1)=192.
\]
Equation \eqref{eq:dyadic-shell} then returns $3$ from $(2,4]$ and
$35=5\cdot7$ from $(4,8]$; the plus extractor subsequently separates $5$
and $7$.

The formulas in this section use the factorization of the index to explain
what the valuation returns, but their displayed inputs are factorials,
binomial coefficients, and gcds. Their cost is a separate issue: evaluating
the resulting cyclotomic integers can be far more expensive than finding the
primes by standard methods.

\section{Archimedean scale separation and factor peeling}
\label{sec:logarithmic}

The local identities inspect the low binary digits of $\Phi_n(2)$.
Its real logarithm offers another way to recover the index arithmetic. The
main term is the degree $\varphi(n)$; the information we need lies in the
small discrepancy from that integer. To use the discrepancy, we must identify
its leading term and control the combined effect of all later terms.

For $t\ge1$, define the elementary correction
\begin{equation}
\label{eq:lambda}
\lambda(t):=-\log_2(1-2^{-t}).
\end{equation}
For $u>0$, define the scale map
\begin{equation}
\label{eq:scale}
\Scale(u):=-\log_2\!\left((\log 2)u\right).
\end{equation}
The scale map approximately inverts the correction, since
\[
(\log 2)\lambda(t)=-\log(1-2^{-t})
=2^{-t}+O(2^{-2t}).
\]
Thus a residual dominated by $\lambda(t)$ should identify $t$ after taking
$\Scale$ and rounding. The estimates below make this statement exact.

Products of already recovered primes also contribute to the residual. We
keep these contributions together in a finite sum, which we call the
logarithmic fingerprint. For a finite set of primes $P$ and an integer
$a\ge1$, define
\begin{equation}
\label{eq:fingerprint}
\Fingerprint_a(P)
:=
\sum_{S\subseteq P}(-1)^{|S|}
\lambda\!\left(a\prod_{p\in S}p\right),
\end{equation}
where the empty product is $1$.

\begin{proposition}[Exact logarithmic fingerprint]
\label{prop:log-master}
Let $n>1$, let $P$ be its set of distinct prime divisors, and put
$a=n/\rad(n)$. Then
\begin{equation}
\label{eq:log-master}
\log_2\Phi_n(2)
=
\varphi(n)-\mu(\rad(n))\Fingerprint_a(P).
\end{equation}
Equivalently,
\begin{equation}
\label{eq:log-divisor}
\log_2\Phi_n(2)
=
\varphi(n)
+
\sum_{d\mid\rad(n)}
\mu\!\left(\frac{\rad(n)}d\right)
\log_2\!\left(1-2^{-ad}\right).
\end{equation}
\end{proposition}

\begin{proof}
The M\"obius product and radical reduction give
\[
\Phi_n(2)
=
\prod_{d\mid\rad(n)}(2^{ad}-1)^{\mu(\rad(n)/d)}.
\]
Factoring $2^{ad}$ from each term produces the leading exponent
$a\varphi(\rad(n))=\varphi(n)$. Since $\rad(n)$ is squarefree,
$\mu(\rad(n)/d)=\mu(\rad(n))\mu(d)$, and its divisors correspond to subsets
of $P$. Taking base-$2$ logarithms gives both displayed forms.
\end{proof}

The empty subset contributes $\lambda(a)$; singleton subsets contribute
$-\lambda(ap)$; larger subsets contribute at products of primes. This order
explains the reconstruction strategy. First recover the integer baseline and the sign of the discrepancy, then
remove all contributions supported on primes already found. The
least unrecovered prime is the smallest remaining subset product. A bound
for the tail is what turns this ordering into an exact rounding rule.

\begin{lemma}[Scale separation]
\label{lem:scale-separation}
For integers $t,s\ge1$,
\begin{equation}
\label{eq:lambda-tail}
\sum_{j\ge1}\lambda(t+js)
\le
\frac{\lambda(t)}{2^s-1},
\end{equation}
and
\begin{equation}
\label{eq:lambda-basic}
2^{-t}
<
(\log 2)\lambda(t)
<
\frac{2^{-t}}{1-2^{-t}}.
\end{equation}
\end{lemma}

\begin{proof}
The expansion
\[
\lambda(t)=\frac1{\log 2}\sum_{k\ge1}\frac{2^{-tk}}k
\]
gives
\[
\sum_{j\ge1}\lambda(t+js)
=\frac1{\log2}\sum_{k\ge1}
\frac{2^{-tk}}{k(2^{sk}-1)}
\le\frac{\lambda(t)}{2^s-1}.
\]
The inequality $x<-\log(1-x)<x/(1-x)$ for $0<x<1$ gives
\eqref{eq:lambda-basic}.
\end{proof}

Write $\round(x)$ for the nearest integer when it is unique, and
$\OddRound(x)$ and $\EvenRound(x)$ for the nearest odd and even integers,
respectively, when unique. Our bounds place each scale strictly within its
rounding interval, so no tie-breaking convention is needed.

\subsection{Nonsquarefree indices: recovering the initial scale}

When $a=n/\rad(n)\ge2$, the entire correction is smaller than $1/2$.
Ordinary rounding can therefore recover $\varphi(n)$ before any prime has
been found. The size of the remaining discrepancy then identifies $a$.

\begin{lemma}[Positivity of the nonsquarefree fingerprint]
\label{lem:fingerprint-positive}
If $a\ge2$ and $P$ is a finite set of primes, then
\begin{equation}
\label{eq:fingerprint-bounds}
\left(1-\frac1{2^a-1}\right)\lambda(a)
\le
\Fingerprint_a(P)
\le
\lambda(a).
\end{equation}
In particular, $\Fingerprint_a(P)>0$.
\end{lemma}

\begin{proof}
Every nonempty subset in \eqref{eq:fingerprint} has product at least $2$, so
\[
\left|\Fingerprint_a(P)-\lambda(a)\right|
\le
\sum_{d\ge2}\lambda(ad)
\le
\frac{\lambda(a)}{2^a-1}.
\]
This gives the lower bound. For the upper bound, adjoining a prime $p$ changes
$\Fingerprint_a(P)$ to
$\Fingerprint_a(P)-\Fingerprint_{ap}(P)$, and the lower bound just proved
shows that the subtracted quantity is positive. Induction on $|P|$ completes
the proof.
\end{proof}

\begin{theorem}[Archimedean recovery of the repeated-power scale]
\label{thm:arch-power}
Let $n>1$ be nonsquarefree and put $L=\log_2\Phi_n(2)$. Then
\begin{equation}
\label{eq:nonsquare-initial}
\round(L)=\varphi(n),
\qquad
\operatorname{sgn}(\round(L)-L)=\mu(\rad(n)),
\end{equation}
and
\begin{equation}
\label{eq:nonsquare-a}
\frac{n}{\rad(n)}
=
\round\!\left(\Scale\!\left(|\round(L)-L|\right)\right).
\end{equation}
\end{theorem}

\begin{proof}
Put $a=n/\rad(n)\ge2$ and let $P$ be the prime-divisor set. By
\cref{prop:log-master,lem:fingerprint-positive},
\[
\varphi(n)-L=\mu(\rad(n))\Fingerprint_a(P),
\qquad
0<\Fingerprint_a(P)\le\lambda(a)\le\lambda(2)<\frac12.
\]
This proves \eqref{eq:nonsquare-initial}.

The bounds in \eqref{eq:fingerprint-bounds} and
\eqref{eq:lambda-basic} imply
$|\Scale(\Fingerprint_a(P))-a|<1/2$. For $a\ge3$ this follows directly from
\[
a+\log_2(1-2^{-a})
<\Scale(\Fingerprint_a(P))
<a-\log_2\!\left(1-\frac1{2^a-1}\right).
\]
For $a=2$, the lower estimate exceeds $3/2$. Also,
$-\log(1-1/4)>1/4+1/32=9/32$, so
\[
(\log2)\Fingerprint_2(P)
\ge\tfrac23(\log2)\lambda(2)
>\frac3{16}>2^{-5/2}.
\]
This gives the upper estimate $\Scale(\Fingerprint_2(P))<5/2$ and proves
\eqref{eq:nonsquare-a}.
\end{proof}

Knowing $a$ fixes the spacing of all subsequent exponents. At each step we
compare the full fingerprint, obtained from $L$, with the fingerprint of the
primes already recovered. Their difference starts at the next prime.

\begin{theorem}[Complete peeling for nonsquarefree indices]
\label{thm:nonsquare-peeling}
Let $n>1$ be nonsquarefree, put $a=n/\rad(n)$ and
$L=\log_2\Phi_n(2)$, and list the distinct prime divisors as
$p_1<\cdots<p_k$. After $p_1,\ldots,p_j$ have been recovered, where
$0\le j<k$, the next prime is
\begin{equation}
\label{eq:nonsquare-peeling}
p_{j+1}
=
\round\!\left(
\frac{\Scale\!\left(
\Fingerprint_a(\{p_1,\ldots,p_j\})-|\round(L)-L|
\right)}a
\right),
\end{equation}
where the set is empty when $j=0$. For a given $n$, iteration stops when
the product of the recovered primes equals $n/a$. Their multiplicities
are then obtained by exact division of $n$.
\end{theorem}

\begin{proof}
Let $P=\{p_1,\ldots,p_k\}$. By \cref{prop:log-master,thm:arch-power},
\[
|\round(L)-L|=\Fingerprint_a(P).
\]
Put $P_j=\{p_1,\ldots,p_j\}$ and write the residual as
\[
R=\Fingerprint_a(P_j)-\Fingerprint_a(P)
=-\sum_{\substack{S\subseteq P\\S\not\subseteq P_j}}
(-1)^{|S|}\lambda\!\left(a\prod_{p\in S}p\right).
\]
Unique factorization makes the subset products distinct. The smallest is
$p_{j+1}$, arising from a singleton with positive coefficient; every other
exponent is at least $ap_{j+1}+a$. With $t=ap_{j+1}\ge4$,
\eqref{eq:lambda-tail} gives
\[
\frac23\lambda(t)\le R\le\frac43\lambda(t).
\]
In particular, $R>0$. By \eqref{eq:lambda-basic},
\[
t+\log_2(1-2^{-t})-\log_2(4/3)
<\Scale(R)<t+\log_2(3/2).
\]
For $t\ge4$, the lower endpoint is at least
$t+\log_2(45/64)>t-1$, and the upper endpoint is less than $t+1$.
Dividing by $a\ge2$ places $\Scale(R)/a$ strictly within distance $1/2$ of
$p_{j+1}$, proving \eqref{eq:nonsquare-peeling}.
\end{proof}

At $n=45$, for example, $n/\rad(n)=3$. The local identities give
\[
\nu_2(\Phi_{45}(2)-1)=3,
\qquad
\nu_2(\Phi_{45}(2)+1)=1.
\]
The real value tells the same story in stages:
$\round(\log_2\Phi_{45}(2))=\varphi(45)=24$; its first residual scale is $3$;
after that scale is removed, the exponents $9$ and $15$ return the prime
divisors $3$ and $5$ after division by $a=3$.

\begin{corollary}[Archimedean primorial identity]
\label{cor:arch-primorial}
Let $m\ge4$, put
\[
L=\log_2\Phi_{m!}(2),
\qquad
M=\round(L).
\]
Then
\begin{equation}
\label{eq:arch-primorial}
\prod_{p\le m}p
=
\frac{m!}{\round\!\left(\Scale\!\left(|L-M|\right)\right)}.
\end{equation}
\end{corollary}

\begin{proof}
Apply \cref{thm:arch-power} to $m!$ and use
$\rad(m!)=\prod_{p\le m}p$.
\end{proof}

\subsection{Odd squarefree indices: removing the unit correction}

For a squarefree index, $a=1$ and the empty-subset term is
$\lambda(1)=1$. We must account for this unit correction before rounding.
When all prime divisors are odd, the remaining exponents are odd as well;
the resulting gap of at least $2$ leaves enough room to identify the least
prime by rounding to the nearest odd integer.

\begin{lemma}[Least odd squarefree scale]
\label{lem:odd-scale}
Let $P$ be a nonempty finite set of odd primes and let $p=\min P$. Then
\begin{equation}
\label{eq:odd-scale-bound}
\left|1-\Fingerprint_1(P)-\lambda(p)\right|
\le
\frac13\lambda(p).
\end{equation}
Moreover,
\begin{equation}
\label{eq:odd-scale-positive}
0<1-\Fingerprint_1(P)<\frac12.
\end{equation}
\end{lemma}

\begin{proof}
The singleton $\{p\}$ contributes $\lambda(p)$ to
$1-\Fingerprint_1(P)$. Every other nonempty subset product is a distinct odd
integer at least $p+2$. Hence \cref{lem:scale-separation} gives
\[
\left|1-\Fingerprint_1(P)-\lambda(p)\right|
\le
\sum_{j\ge1}\lambda(p+2j)
\le
\frac13\lambda(p).
\]
The lower bound is therefore positive. Since $p\ge3$, the upper bound is at
most $(4/3)\lambda(3)<1/2$.
\end{proof}

\begin{theorem}[Odd squarefree logarithmic extraction]
\label{thm:odd-squarefree}
Let $n>1$ be odd and squarefree and put $L=\log_2\Phi_n(2)$. Then
\begin{equation}
\label{eq:odd-initial}
\round(L)=\varphi(n)-\mu(n),
\qquad
\operatorname{sgn}(L-\round(L))=\mu(n).
\end{equation}
\begin{equation}
\label{eq:odd-lpf}
\lpf(n)
=
\OddRound\!\left(
\Scale\!\left(|L-\round(L)|\right)
\right).
\end{equation}
\end{theorem}

\begin{proof}
Let $P$ be the prime-divisor set. Since $\lambda(1)=1$,
\cref{prop:log-master} gives
\[
L
=
\varphi(n)-\mu(n)
+\mu(n)\bigl(1-\Fingerprint_1(P)\bigr).
\]
By \cref{lem:odd-scale}, the final factor is positive and smaller than $1/2$,
which proves \eqref{eq:odd-initial}. If $p=\lpf(n)$, the same lemma gives
\[
\frac23\lambda(p)
\le
|L-\round(L)|
\le
\frac43\lambda(p).
\]
Together with \eqref{eq:lambda-basic}, this yields
\[
p+\log_2(1-2^{-p})-\log_2(4/3)
<
\Scale(|L-\round(L)|)
<
p+\log_2(3/2).
\]
For $p\ge3$, these endpoints lie strictly between $p-1$ and $p+1$.
Therefore $p$ is the unique odd integer within distance $1$ of the scale,
proving \eqref{eq:odd-lpf}.
\end{proof}

After the first prime has been found, the next step must account for all
subset products involving the recovered primes. The fingerprint does exactly
this, so the same estimate can be applied to each successive residual.

\begin{theorem}[Complete peeling for odd squarefree indices]
\label{thm:odd-peeling}
Let $n>1$ be odd and squarefree, put $L=\log_2\Phi_n(2)$, and list its prime
divisors as $p_1<\cdots<p_k$. After $p_1,\ldots,p_j$ have been recovered,
where $0\le j<k$, one has
\begin{equation}
\label{eq:odd-peeling}
p_{j+1}
=
\OddRound\!\left(
\Scale\!\left(
|L-\round(L)|-1+\Fingerprint_1(\{p_1,\ldots,p_j\})
\right)
\right),
\end{equation}
where the set is empty when $j=0$. For a given $n$, the process stops when
the product of the recovered primes equals $n$.
\end{theorem}

\begin{proof}
Let $P=\{p_1,\ldots,p_k\}$. By \cref{prop:log-master,thm:odd-squarefree},
the residual inside \eqref{eq:odd-peeling} is
\[
\Fingerprint_1(\{p_1,\ldots,p_j\})-\Fingerprint_1(P).
\]
Its unique smallest scale is $p_{j+1}$; every other exponent is a distinct
odd integer at least $p_{j+1}+2$. The proof of \cref{lem:odd-scale} therefore
places the residual between $(2/3)\lambda(p_{j+1})$ and
$(4/3)\lambda(p_{j+1})$. The argument used in
\cref{thm:odd-squarefree} then returns $p_{j+1}$ by odd rounding.
\end{proof}

For $n=105=3\cdot5\cdot7$, one has
\[
\round\!\left(\log_2\Phi_{105}(2)\right)
=
\varphi(105)-\mu(105)=49.
\]
The residual is negative, hence $\mu(105)=-1$, and successive odd scales
return $3$, $5$, and $7$. In particular, after $3$ and $5$ have been found,
their contribution includes $\lambda(15)$ as well as the two singleton
terms. Retaining this mixed term is essential to the next subtraction.

\subsection{Even squarefree indices}
\label{sec:even-squarefree}

The prime $2$ changes the first logarithmic correction. Pairing the terms
whose subset products differ by a factor of $2$ replaces $\lambda$ by
\begin{equation}
\label{eq:kappa}
\kappa(t):=\lambda(t)-\lambda(2t)=\log_2(1+2^{-t}).
\end{equation}
For a finite set $P$ of odd primes, write
\begin{equation}
\label{eq:even-fingerprint}
\mathcal G(P):=\sum_{S\subseteq P}(-1)^{|S|}
\kappa\!\left(\prod_{p\in S}p\right).
\end{equation}
Thus $\Fingerprint_1(\{2\}\cup P)=\mathcal G(P)$. The resulting correction
can exceed $1/2$, so ordinary nearest-integer rounding need not recover
$\varphi(n)$. The fact that $\varphi(n)$ is even for $n>2$ supplies the correct
rounding rule.

\begin{theorem}[Even squarefree logarithmic extraction and peeling]
\label{thm:even-squarefree}
Let $n=2m>2$ be squarefree, put $L=\log_2\Phi_n(2)$, and list the prime
divisors of $m$ as $p_1<\cdots<p_k$. Then
\begin{equation}
\label{eq:even-initial}
M:=\EvenRound(L)=\varphi(n),
\qquad
\operatorname{sgn}(M-L)=\mu(n).
\end{equation}
For $P_j=\{p_1,\ldots,p_j\}$, with $P_0=\varnothing$ and $0\le j<k$,
\begin{equation}
\label{eq:even-peeling}
p_{j+1}
=\OddRound\!\left(
\Scale\!\left(\mathcal G(P_j)-|M-L|\right)
\right).
\end{equation}
The process stops when the product of the recovered odd primes equals $m$.
\end{theorem}

\begin{proof}
Put $P=\{p_1,\ldots,p_k\}$. Splitting the subsets in
\eqref{eq:fingerprint} according to whether they contain $2$, and applying
\cref{prop:log-master}, gives
\[
L=\varphi(n)-\mu(n)\mathcal G(P).
\]
Every nonempty subset product in \eqref{eq:even-fingerprint} is a distinct
odd integer at least $3$. Since
$(\log 2)\kappa(t)=\log(1+2^{-t})<2^{-t}$,
\[
\left|\mathcal G(P)-\log_2(3/2)\right|
\le\sum_{j\ge0}\kappa(3+2j)
<\frac{1}{6\log 2}.
\]
The resulting interval lies in $(0,1)$: both $\log(3/2)$ and
$\log(4/3)$ exceed $1/6$. Hence $0<\mathcal G(P)<1$.
Since $\varphi(n)$ is even for $n>2$, this proves
\eqref{eq:even-initial} and $|M-L|=\mathcal G(P)$.

For the peeling step put $q=p_{j+1}$ and
$R=\mathcal G(P_j)-\mathcal G(P)$. Subtraction cancels the terms supported
on $P_j$. Unique factorization makes the remaining subset products
distinct; the smallest is $q$, with coefficient $+1$. Consequently,
\[
R=\kappa(q)+
\sum_{\substack{d\ge q+2\\d\text{ odd}}}\varepsilon_d\kappa(d),
\qquad \varepsilon_d\in\{-1,0,1\},
\]
where only finitely many coefficients are nonzero. Let $x=2^{-q}\le1/8$.
The tail has absolute value less than $x/(3\log 2)$, while
\[
\frac{x}{1+x}<\log(1+x)<x.
\]
It follows that
\begin{equation}
\label{eq:even-residual-bounds}
\frac59\,2^{-q}<(\log 2)R<\frac43\,2^{-q},
\end{equation}
and therefore
\[
q-\log_2(4/3)<\Scale(R)<q+\log_2(9/5).
\]
This interval lies strictly inside $(q-1,q+1)$, which proves
\eqref{eq:even-peeling}.
\end{proof}

For example, $\Phi_{10}(2)=11$ and
$\mathcal G(\{5\})=\log_2(3/2)-\log_2(33/32)$.
Ordinary rounding sends $\log_2 11$ to $3$, whereas even rounding returns
$\varphi(10)=4$. The first peeling residual is $\kappa(5)$, whose scale
returns the missing odd factor $5$.

Together, the three cases give logarithmic extraction formulas for every
given index $n>2$. The remaining index $n=2$ is elementary, but its logarithmic value
cannot be distinguished from that of $n=6$:
$\Phi_2(2)=\Phi_6(2)=3$.

\subsection{Prime sets and required precision}

The factorial and interval indices used in the local theory also fall into
the real reconstruction regimes. Factorial indices are nonsquarefree once
the argument is at least $4$; the interval products $R_n$ are odd and
squarefree for $n\ge2$. Substitution therefore gives real counterparts of
the earlier primorial and successor identities.

\begin{corollary}[Prime recovery from one factorial value]
\label{cor:factorial-log}
Let $m\ge4$. The single exact real number $\log_2\Phi_{m!}(2)$, together with
$m$, determines every prime at most $m$. Equation
\eqref{eq:arch-primorial} first determines their product, and
\eqref{eq:nonsquare-peeling} then returns the primes in increasing order. The
process stops when the product of the recovered primes equals the value given
by \eqref{eq:arch-primorial}.
\end{corollary}

\begin{proof}
Apply \cref{cor:arch-primorial,thm:nonsquare-peeling} to $m!$.
\end{proof}

\begin{corollary}[Archimedean successor-prime identity]
\label{cor:arch-next}
For $n\ge2$, put
\[
L=\log_2\Phi_{R_n}(2),
\qquad
M=\round(L).
\]
Then
\begin{equation}
\label{eq:arch-next}
\NextPrime(n)
=
\OddRound\!\left(\Scale\!\left(|L-M|\right)\right).
\end{equation}
\end{corollary}

\begin{proof}
By \cref{lem:Rn}, $R_n$ is odd and squarefree for $n\ge2$, and its least
prime factor is $\NextPrime(n)$. Apply \cref{thm:odd-squarefree}.
\end{proof}

\begin{remark}[Precision]
\label{rem:precision}
The extraction formulas use exact real values. Recovering a prime $p$
requires resolving a residual of order $2^{-t}$, where $t=ap$ in the
nonsquarefree case and $t=p$ in either squarefree case. Thus the required
number of fractional binary digits grows with $t$, with additional guard
bits for arithmetic error. Representing $L$ also requires
$O(\log\varphi(n))$ bits for its integer part.

For an exact prime-peeling residual $R$, all three cases satisfy
\[
\frac59<(\log 2)R\,2^t<\frac{32}{21}.
\]
Consequently, a sufficient error tolerance valid in every case is
\[
|\widetilde R-R|<\frac{2^{-t}}{18\log 2}.
\]
This keeps $(\log 2)\widetilde R\,2^t$ strictly between $1/2$ and $2$, so
$|\Scale(\widetilde R)-t|<1$ and the factor-rounding decision is unchanged.
The condition concerns the total residual error, including all subtractions,
and assumes that the initial rounding and preceding prime recoveries were
correct. Direct fingerprint evaluation may also involve exponentially many
subsets. The formulas are reconstruction identities rather than efficient
factorization algorithms.
\end{remark}

\section{The infinite limit and its Mellin family}
\label{sec:infinite-prime-constant}

The proof of odd squarefree peeling uses the least missing prime and a
bound over all larger odd exponents. It does not depend on how many prime
factors remain. We can therefore ask whether the same rule still works when
the finite support is replaced by the set of all odd primes.

There are two points to check: the fingerprints must have a limit, and the
infinite tail must leave the same rounding margin as a finite tail. Absolute
convergence handles the first; the estimate already used for finite peeling
handles the second. Define
\begin{equation}
\label{eq:Omega-def}
\Omega
:=
-\sum_{\substack{d>1\\d\text{ odd}}}\mu(d)\lambda(d)
=
0.25061403238015047217999\ldots.
\end{equation}

Let $3=q_1<q_2<\cdots$ be the odd primes and put
\[
Q_k=\prod_{j=1}^k q_j.
\]

\begin{proposition}[Convergence, product, and normalized cyclotomic limits]
\label{prop:Omega}
The series in \eqref{eq:Omega-def} converges absolutely, and
\begin{equation}
\label{eq:Omega-product}
2^\Omega
=
\prod_{\substack{d>1\\d\text{ odd}}}(1-2^{-d})^{\mu(d)}.
\end{equation}
Moreover,
\begin{equation}
\label{eq:Omega-limit}
\Omega
=
1+
\lim_{k\to\infty}(-1)^k
\left(
\log_2\Phi_{Q_k}(2)-\varphi(Q_k)
\right).
\end{equation}
Equivalently,
\begin{equation}
\label{eq:Omega-normalized-limit}
2^\Omega
=
2\lim_{k\to\infty}
\left(
\frac{\Phi_{Q_k}(2)}{2^{\varphi(Q_k)}}
\right)^{(-1)^k},
\end{equation}
and
\begin{equation}
\label{eq:Omega-half-limit}
2^\Omega
=
2\lim_{k\to\infty}
\Phi_{Q_k}\!\left(\frac12\right)^{(-1)^k}.
\end{equation}
\end{proposition}

\begin{proof}
By \eqref{eq:lambda-basic}, $\lambda(d)\ll2^{-d}$, proving absolute
convergence. Since $-\lambda(d)=\log_2(1-2^{-d})$, exponentiation gives
\eqref{eq:Omega-product}.

For $Q_k$, \cref{prop:log-master} gives
\[
1+(-1)^k
\left(
\log_2\Phi_{Q_k}(2)-\varphi(Q_k)
\right)
=
-\sum_{\substack{d\mid Q_k\\d>1}}\mu(d)\lambda(d).
\]
Absolute convergence permits passage to the limit and proves
\eqref{eq:Omega-limit}. Taking base-$2$ exponentials gives
\[
\lim_{k\to\infty}
\left(
\frac{\Phi_{Q_k}(2)}{2^{\varphi(Q_k)}}
\right)^{(-1)^k}
=
2^{\Omega-1},
\]
which proves \eqref{eq:Omega-normalized-limit}. Finally, cyclotomic
reciprocity gives
\[
\Phi_m(2)=2^{\varphi(m)}\Phi_m\!\left(\frac12\right)
\qquad(m>1),
\]
and \eqref{eq:Omega-half-limit} follows.
\end{proof}

Cyclotomic reciprocity explains the appearance of $1/2$: dividing the value
at $2$ by $2^{\varphi(m)}$ gives the reciprocal evaluation exactly. The
factor $(-1)^k$ compensates for the alternating M\"obius sign of the odd
primorials. Thus the normalization selects the same logarithmic correction
throughout the sequence. We next quantify its convergence before decoding
the limit.

\begin{proposition}[Certified approximation to the prime constant]
\label{prop:Omega-error}
For odd $D\ge3$, put
\[
\Omega_{\le D}:=-\sum_{\substack{3\le d\le D\\d\text{ odd}}}
\mu(d)\lambda(d).
\]
Then
\begin{equation}
\label{eq:Omega-truncation}
|\Omega-\Omega_{\le D}|\le\frac43\lambda(D+2).
\end{equation}
The odd-primorial approximants
\[
\Omega_k:=1-\Fingerprint_1(\{q_1,\ldots,q_k\}),
\qquad \Omega_0:=0,
\]
increase strictly to $\Omega$, and their error satisfies
\begin{equation}
\label{eq:Omega-primorial-error}
\frac23\lambda(q_{k+1})
\le\Omega-\Omega_k
\le\frac43\lambda(q_{k+1})
\qquad(k\ge0).
\end{equation}
\end{proposition}

\begin{proof}
The omitted terms in $\Omega_{\le D}$ have distinct odd indices at least $D+2$.
By \cref{lem:scale-separation},
\[
|\Omega-\Omega_{\le D}|
\le\sum_{j\ge0}\lambda(D+2+2j)
\le\frac43\lambda(D+2).
\]
For $\Omega_k$, absolute convergence permits cancellation of every term
supported on $\{q_1,\ldots,q_k\}$. The smallest remaining index is
$q_{k+1}$, with positive coefficient, and every other remaining index is an
odd integer at least $q_{k+1}+2$. The same tail estimate gives
\eqref{eq:Omega-primorial-error}. Finally, the subset definition gives
\[
\Omega_{k+1}-\Omega_k
=\Fingerprint_{q_{k+1}}(\{q_1,\ldots,q_k\})>0,
\]
where positivity follows from \cref{lem:fingerprint-positive}.
\end{proof}

Taking $D=101$ in \eqref{eq:Omega-truncation} gives an error smaller than
$1.90\times10^{-31}$. Thus the displayed digits of $\Omega$ can be checked
by evaluating a short finite sum with adequate arithmetic precision.
Unlike truncation by the integer cutoff $D$, truncation by prime support has
the monotonicity in \eqref{eq:Omega-primorial-error}: each new prime makes a
positive correction, and the first omitted prime sets the error scale.

\begin{theorem}[Infinite prime peeling]
\label{thm:infinite-peeling}
With $q_1,q_2,\ldots$ as above, for every $j\ge0$,
\begin{equation}
\label{eq:infinite-peeling}
q_{j+1}
=
\OddRound\!\left(
\Scale\!\left(
\Omega-1+\Fingerprint_1(\{q_1,\ldots,q_j\})
\right)
\right),
\end{equation}
where the set is empty when $j=0$. Thus $\Omega$, together with the prime
$2$, recursively determines all primes.
\end{theorem}

\begin{proof}
The residual in \eqref{eq:infinite-peeling} is exactly $\Omega-\Omega_j$.
By \cref{prop:Omega-error}, it lies between
$(2/3)\lambda(q_{j+1})$ and $(4/3)\lambda(q_{j+1})$.
The strict rounding interval established in
\cref{thm:odd-squarefree} therefore gives \eqref{eq:infinite-peeling}.
\end{proof}

The first steps show how the infinite decoder continues the finite one:
\[
\Scale(\Omega)=2.525\ldots,
\qquad
\Scale(\Omega-\lambda(3))=4.637\ldots,
\]
which return $3$ and $5$. Subtracting
$\lambda(3)+\lambda(5)-\lambda(15)$ gives a scale of
$6.884\ldots$, which returns $7$.

\subsection{The Mellin family behind the constant}

The constant $\Omega$ uses one fixed decay rate, $2^{-d}$. To understand
the role of that choice, replace it by $e^{-dt}$ and allow $t$ to vary.
This gives a family in which the integer $d$ acts by dilation of the real
variable. The Mellin transform is suited to this dependence: it converts
such dilations into powers $d^{-s}$.

For $t>0$, define
\begin{equation}
\label{eq:L-def}
\mathcal L(t)
:=
-\sum_{\substack{m\ge1\\m\text{ odd}}}
\mu(m)\log(1-e^{-mt}).
\end{equation}
At the original binary scale,
\begin{equation}
\label{eq:L-Omega}
\mathcal L(\log 2)=(1-\Omega)\log 2.
\end{equation}
For each fixed $t>0$, the defining series converges absolutely by geometric
decay. In the Mellin transform, the logarithmic series supplies
$\zeta(s+1)$ and the odd M\"obius sum supplies the reciprocal Euler product
below.

\begin{proposition}[Mellin transform of the odd M\"obius product]
\label{prop:Mellin}
For $\Re s>1$,
\begin{equation}
\label{eq:Mellin}
\int_0^\infty \mathcal L(t)t^{s-1}\,dt
=
\frac{\Gamma(s)\zeta(s+1)}
{\zeta(s)(1-2^{-s})}.
\end{equation}
The expression on the right gives a meromorphic continuation of the transform
to the complex plane.
\end{proposition}

\begin{proof}
Let $\sigma=\Re s>1$. The expansion
\[
-\log(1-e^{-mt})=\sum_{r\ge1}\frac{e^{-mrt}}r
\]
may be inserted into \eqref{eq:L-def} and integrated term by term, because
\begin{align*}
&\sum_{\substack{m\ge1\\m\text{ odd}}}|\mu(m)|
\sum_{r\ge1}\frac1r
\int_0^\infty e^{-mrt}t^{\sigma-1}\,dt\\
&\hspace{25mm}=
\Gamma(\sigma)
\left(\sum_{\substack{m\ge1\\m\text{ odd}}}
\frac{|\mu(m)|}{m^\sigma}\right)
\left(\sum_{r\ge1}\frac1{r^{\sigma+1}}\right)
<\infty.
\end{align*}
It follows that
\[
\int_0^\infty \mathcal L(t)t^{s-1}\,dt
=
\Gamma(s)\zeta(s+1)
\sum_{\substack{m\ge1\\m\text{ odd}}}\frac{\mu(m)}{m^s}.
\]
The odd M\"obius Dirichlet series has Euler product
\[
\sum_{\substack{m\ge1\\m\text{ odd}}}\frac{\mu(m)}{m^s}
=
\prod_{p\ne2}(1-p^{-s})
=
\frac1{\zeta(s)(1-2^{-s})},
\]
which proves \eqref{eq:Mellin}. Meromorphic continuation follows from the
right-hand side.
\end{proof}

\begin{remark}[Including the prime $2$]
\label{rem:all-prime-constant}
The all-prime analogue is
\[
\widetilde\Omega
:=-\sum_{d>1}\mu(d)\lambda(d)
=\Omega+\frac{\mathcal L(2\log 2)}{\log 2}
=0.64143346289118272006\ldots.
\]
To see the relation, separate the odd and even squarefree indices. Since
$\mu(2m)=-\mu(m)$ for odd $m$, the corresponding family is
\[
\mathcal L_{\mathrm{all}}(t)
:=-\sum_{m\ge1}\mu(m)\log(1-e^{-mt})
=\mathcal L(t)-\mathcal L(2t).
\]
Thus $\mathcal L_{\mathrm{all}}(\log 2)=(1-\widetilde\Omega)\log 2$, and
\cref{prop:Mellin} gives the simpler transform
\[
\int_0^\infty\mathcal L_{\mathrm{all}}(t)t^{s-1}\,dt
=\frac{\Gamma(s)\zeta(s+1)}{\zeta(s)}
\qquad(\Re s>1).
\]
We retain $\Omega$ as the main constant because it continues the finite
odd-squarefree decoder with the same kernel $\lambda$ and the same
nearest-odd rounding rule at every step. Including $2$ does not by itself
give a uniform initial step: ordinary rounding of
$\Scale(\widetilde\Omega)=1.169\ldots$ returns $1$. After taking $2$ as
the initial prime, the all-prime variant can instead be decoded using the
$\kappa$-fingerprint of \cref{sec:even-squarefree}; absolute convergence
and the bound \eqref{eq:even-residual-bounds} justify passage to the limit.
\end{remark}

\section{Concluding remarks}

The initial binary observation leads to two forms of explicit recovery.
The local formulas obtain the repeated-power quotient from
$\Phi_n(2)-1$ and the least prime factor of a squarefree index from
$\Phi_n(2)+1$. The real formulas instead inspect the discrepancy between
$\log_2\Phi_n(2)$ and its appropriate integer baseline. In both settings,
the M\"obius product identifies the first relevant divisor; the essential
additional work over the reals is to keep the remaining tail inside a strict
rounding interval.

This tail control also explains why finite prime recovery survives passage
to the odd-primorial limit. The next prime remains the dominant missing
contribution even when infinitely many later primes are present. The
resulting constant admits both a cyclotomic limit description and an explicit
recursive decoder, while its variable-scale version has the Mellin transform
\eqref{eq:Mellin}. These identities describe what the exact values retain.
The precision bounds specify the separate cost of resolving that information:
a prime at exponent $t$ requires a residual known to absolute accuracy on the
scale of $2^{-t}$.

\paragraph{AI assistance.}
The mathematical ideas are the author's. GPT~6 assisted with manuscript
preparation and the writing of proofs.

\small
\setlength{\bibsep}{0.25em}

\end{document}